\documentclass[reqno,11pt]{amsart}
\usepackage{amssymb}
\usepackage{amsthm}
\usepackage{amsbsy}
\usepackage{amsmath}
\usepackage{mathrsfs}
\usepackage[pdftex]{color}
\usepackage{enumerate}
\usepackage[pdftex,pdftitle={Brownian semistationary processes and conditional full support},pdfauthor={Mikko S. Pakkanen}]{hyperref}

\newcommand{\ud}{\mathrm{d}}
\newcommand{\proba}{\mathbf{P}}

\newcommand{\R}{\mathbb{R}}

\newcommand{\eqdefr}{=\mathrel{\mathop:}}
\newcommand{\eqdefl}{\mathrel{\mathop:}=}

\newcommand{\ut}{\underline{t}}
\newcommand{\bss}{\mathcal{BSS}}

\numberwithin{equation}{section}

\theoremstyle{plain}

\newtheorem{thm}[equation]{Theorem}
\newtheorem{lem}[equation]{Lemma}
\newtheorem{cor}[equation]{Corollary}

\theoremstyle{definition}

\newtheorem{rem}[equation]{Remark}

\begin{document}

\title[Brownian semistationary processes]{Brownian semistationary processes and conditional full support}

\author{Mikko S.\ Pakkanen}

\address{
Department of Mathematics and Statistics \\
University of Helsinki\\
P.O.\ Box 68\\
FI-00014 Helsingin yliopisto \\
Finland}
\email{msp@iki.fi}
\urladdr{http://www.iki.fi/msp/}
\date{\today}

\begin{abstract}
In this note, we study the infinite-dimensional conditional laws of Brownian semistationary processes. Motivated by the fact that these processes are typically not semimartingales, we present sufficient conditions ensuring that a Brownian semistationary process has conditional full support, a property introduced by Guasoni, R\'asonyi, and Schachermayer [Ann.\ Appl.\ Probab.,\ 18 (2008) pp.\ 491--520]. By the results of Guasoni, R\'asonyi, and Schachermayer, this property has two important implications. It ensures, firstly, that the process admits no free lunches under proportional transaction costs, and secondly, that it can be approximated pathwise (in the sup norm) by semimartingales that admit equivalent martingale measures.
\end{abstract}


\maketitle

\section{Introduction}

\subsection{Motivation} Originally, Brownian semistationary ($\bss$) processes were introduced by Barndorff-Nielsen and Schmiegel \cite{barnd} as phenomenological models of the behavior of scalar components of turbulent velocity fields. There has been also an emerging interest in applying these processes to the modeling of price dynamics in finance.  
These applications are facilitated by some recently-developed statistical methods for $\bss$ processes (see \cite{barnd2,barnd3}). A generic $\bss$ process $(X_t)_{t \in \R}$ is defined by  
\begin{equation}\label{bssprocess}
X_t \eqdefl \mu +  \int_{-\infty}^t g(t-s) \sigma_s \ud W_s + \int_{-\infty}^t q(t-s) a_s \ud s,
\end{equation}
where $\mu \in \R$ is a constant, $g$ and $q$ are ``suitable'' \emph{memory functions} on $(0,\infty)$, $(\sigma_s)_{s \in \R}$ and $(a_s)_{s \in \R}$ are c\`adl\`ag processes, and $(W_t)_{t \in \R}$ is a standard Brownian motion. The processes $\sigma$ and $a$ are called \emph{intermittency} (or \emph{volatility}) and \emph{drift} processes, respectively.

As discussed in \cite[pp.\ 4--5]{barnd}, in many cases, $\bss$ processes are \emph{not} semimartingales. This might be seen as an issue when these processes are used as models of price dynamics, since the well-known result of Delbaen and Schachermayer \cite[Theorem 7.2]{delba} implies that then there exist approximate arbitrage opportunities (dubbed \emph{free lunches with vanishing risk}) among simple trading strategies. However, for practical purposes, these arbitrage opportunities might not be that relevant, if they are not robust to various ``imperfections'' (e.g.\ illiquidity or transaction costs) present in real financial markets. Indeed, Guasoni, R\'asonyi, and Schachermayer \cite[Theorem 1.2]{guaso2} have recently shown that under arbitrarily small \emph{proportional transaction costs}, even a non-semimartingale is devoid of free lunches, provided that is has a particular distributional property, called \emph{conditional full support}. Thus, to study the applicability of $\bss$ processes to financial modeling, it is of interest to determine whether they have conditional full support.

Aside from this connection to mathematical finance, there are further reasons why the question, whether $\bss$ processes have conditional full support, is of interest. The conditional full support property, which basically entails that at any given time, the (infinite-dimensional) conditional law of the ``future'' of the process, given the ``past'', has the largest possible support, is also a desirable feature for a realistic model of turbulence. Additionally, it is worth mentioning that the aforementioned result of Guasoni, R\'asonyi, and Schachermayer can be stated also in purely probabilistic context---without any reference to mathematical finance. Namely, it ensures that a process with conditional full support can be approximated pathwise (in the sup norm) by \emph{semimartingales} that can be turned into \emph{martingales} by equivalent changes of measures (see \cite[Theorem 2.11]{guaso2}). 

\subsection{Conditional full support}

The conditional full support property is defined as follows. Fix a finite time horizon $T \in (0,\infty)$. For any $x \in \R$, denote by $C_x([u,v])$ the set of functions $f \in C( [u,v])$ such that $f(u) = x$. We say that continuous process $(X_t)_{t \in [0,T]}$ has \emph{conditional full support} (CFS), if for every $\ut \in [0,T)$ and a.e.\ $\omega \in \Omega$,
\begin{equation}\label{cfsdef}
\mathrm{supp}\big(\mathrm{Law}\big[(X_t)_{t \in [\ut,T]}\big| \mathscr{F}_{\ut}\big](\omega)\big) = C_{X_{\ut}(\omega)}([\ut,T]),
\end{equation}
where $\mathrm{Law}\big[(X_t)_{t \in [\ut,T]}\big| \mathscr{F}_{\ut}\big]$ is interpreted as a regular conditional law on $C([\ut,T])$, equipped with the sup norm and the associated Borel $\sigma$-algebra, $(\mathscr{F}_t)_{t \in[0,T]}$ is the natural filtration of $X$, and $\mathrm{supp}(\mu)$ is the \emph{support} of $\mu$, i.e.\ the smallest closed set with $\mu$-measure one. 

Notably, in \cite[p.\ 493]{guaso2} the CFS property was defined for a priori \emph{positive} processes (i.e.\ $C_{X_{\ut}(\omega)}([\ut,T])$ replaced with $C_{X_{\ut}(\omega)}([\ut,T],\R_+)$ in the definition above), but e.g.\ the exponential process $\mathrm{e}^{X_t}$, $t\in[0,T]$ has CFS in the sense of \cite{guaso2} if and only if $X$ has CFS in the sense of \eqref{cfsdef}, see Remark 2.1 of \cite{pakka}.

\subsection{Main result}

The purpose of this paper is to show that, in a finite time interval, the $\bss$ process $(X_t)_{t\in [0,T]}$, as defined by \eqref{bssprocess}, has CFS, when it satisfies the two sets of conditions given below. 

Firstly, to ensure that $X$, and some related processes that appear in the proofs below, are well-defined and have continuous modifications, we introduce the following conditions:
\begin{enumerate}[(i)]
\item\label{c.modif} $\int_{-\infty}^t q(t-s) a_s \ud s$, $t \in [0,T]$ is a.s.\ well-defined and has a continuous modification,
\item\label{c.integ} $\sup_{t \in (-\infty,T]} \mathbf{E}[\sigma^2_t]< \infty$,
\item\label{c.estim} $g \in L^2 ((0,\infty))$ and there exists $\alpha>0$ and $C>0$ such that
\begin{equation*}
\int_0^{\infty} g(s)^2 \ud u - \int_0^{\infty}  g(t+s)g(s) \ud s \leq C t^{\alpha} \quad\textrm{for all $t \in [0,T]$.}
\end{equation*}
\end{enumerate}
In particular, \eqref{c.integ} and \eqref{c.estim} imply that also the Brownian-driven part of $X$ has a continuous modification (by the Kolmogorov--Chentsov criterion).

Secondly, the CFS property clearly cannot hold if $X$ is degenerate. We rule this out through these remaining conditions:
\begin{enumerate}[(i)]\setcounter{enumi}{3}
\item\label{c.decomp} $W = \beta \bar{W} + \sqrt{1-\beta^2}\bar{W}^{\perp}$ for some $\beta \in (-1,1)$, where $\bar{W}$ and $\bar{W}^{\perp}$ are independent standard Brownian motions, such that $\sigma$ and $a$ are independent of $\bar{W}^{\perp}$ and adapted to some filtration with respect to which $\bar{W}$ is a semimartingale,
\item\label{c.snondeg} $\lambda (\{t \in [0,T] : \sigma_t = 0 \})= 0$ a.s.,
\item\label{c.gnondeg} $\int_0^{\varepsilon} |g(s)|\ud s>0$ for all $\varepsilon>0$.
\end{enumerate}
(Here, and in what follows, $\lambda$ stands for the Lebesgue measure on the real line.)
Conditions \eqref{c.snondeg} and \eqref{c.gnondeg} are rather minimal, simply needed to guarantee that the intermittency process $\sigma$ and the memory function $g$ do not vanish on a set with positive measure and near the origin, respectively. Condition \eqref{c.decomp}, which requires existence of a driving Brownian component independent of the drift and intermittency processes, is admittedly more restrictive. However, if it was omitted, the other conditions alone would not suffice, see Remark \ref{dependent} below for further discussion.

To give a concrete example, as pointed out in \cite[pp.\ 4--5]{barnd2}, the functions
\begin{equation*}
g(t) \eqdefl t^{\kappa}\mathrm{e}^{-\rho t}, \quad t \in (0,\infty),
\end{equation*}
where $\kappa \in (-1/2,0) \cup (0,1/2)$ and $\rho>0$, satisfy \eqref{c.estim}, and clearly also \eqref{c.gnondeg}.
Regarding the intermittency process $\sigma$, we note that by Fubini's theorem, \eqref{c.snondeg} holds if $\proba[\sigma_t = 0] =0$ for all $t \in [0,T]$.
In particular, when $\sigma$ is stationary, and its stationary law $\mu_{\sigma}$ satisfies $\mu_{\sigma}(\{0\})=0$ and $\int_{\R} x^2 \mu_{\sigma}(\ud x)<\infty$, clearly both \eqref{c.integ} and \eqref{c.snondeg} hold.

It turns out that when conditions \eqref{c.modif}--\eqref{c.gnondeg} are in force, instead of $\bss$ processes of the \emph{precise} form \eqref{bssprocess}, it is more convenient to establish CFS for slightly more general processes.

\begin{thm}[Conditional full support]\label{cfsforbss}
Let $(Y_t)_{t\in[0,T]}$ be a continuous process, $(\sigma_t)_{t \in (-\infty,T]}$ a c\`adl\`ag process that satisfies conditions \eqref{c.integ} and \eqref{c.snondeg}, $(B_t)_{t \in (-\infty,T]}$ a Brownian motion, and $g : (0,\infty)\rightarrow \R$ a function that satisfies conditions \eqref{c.estim} and \eqref{c.gnondeg}. If $(Y,\sigma)$ is independent of $B$, then
process
\begin{equation}\label{genbss}
Z_t \eqdefl Y_t + \int_{-\infty}^t g(t-s) \sigma_s \ud B_s, \quad t \in [0,T]
\end{equation}
has CFS. 
\end{thm}

Finally, substitutions
\begin{equation*}
Y_t \eqdefl \mu + \beta \int_{-\infty}^{t}  g(t-s) \sigma_s \ud \bar{W}_t + \int_{-\infty}^{t} q(t-s) a_s \ud s, \quad B_t \eqdefl \sqrt{1-\beta^2}\bar{W}^{\perp}_t,
\end{equation*}
yield the promised result:

\begin{cor} If conditions \eqref{c.modif}--\eqref{c.gnondeg} hold, then the $\bss$ process $(X_t)_{t \in [0,T]}$, as defined by \eqref{bssprocess}, has CFS.
\end{cor}

\begin{rem}\label{dependent}
In Theorem \ref{cfsforbss}, it is crucial that $(Y,\sigma)$ is independent of $B$. For instance, in the (semimartingale) special case $g(t) = 1$ for all $t \in [0,T]$, there exist strictly positive integrands $\sigma$ that depend on $B$ and prevent $Z$ from having CFS. As pointed out in \cite[Example 3.10]{pakka}, this is the case e.g.\ when $Y_t \eqdefl 1$ and $\sigma_t \eqdefl \mathrm{e}^{B_t - t/2}\mathbf{1}_{[0,T]}(t)$.   
\end{rem}

\section{Proof of Theorem \ref{cfsforbss}}

We introduce first some notations. In what follows, all random variables and stochastic processes are defined on a complete probability space $(\Omega,\mathscr{F},\proba)$. As usual, $\| \cdot \|_2$, $\| \cdot\|_{\infty}$, and $\| \cdot \|_{\mathrm{op}}$ stand for the $L^2$, sup, and operator norms, respectively. Additionally, for any $t \in [u,v]$, we define $\pi_t : C([u,v]) \rightarrow \R$ to be the evaluation map at $t$, i.e.\ $\pi_t(h) = h(t)$, and we denote by $\phi_d(\,\cdot \, ;\boldsymbol{\mu},\boldsymbol{\Sigma})$ the probabilitity density function of the $d$-dimensional Gaussian distribution with mean $\boldsymbol{\mu}$ and covariance $\boldsymbol{\Sigma}$.

The proof of Theorem \ref{cfsforbss} builds upon the earlier proofs of CFS for Brownian moving averages \cite{chern} and stochastic integrals \cite{pakka}.
Similarly to the ``freezing'' argument used in \cite{pakka}, by a suitable conditioning, we may regard the processes $Y$ and $\sigma$ in \eqref{genbss} as deterministic functions, reducing $Z$ to a Gaussian process. Thus, to establish CFS through this method, we need to characterize the support of this Gaussian process.

\begin{lem}[Support]\label{support}
Let $\ut \in [0,T)$, $g:(0,\infty) \rightarrow \R$ a function that satisfies conditions \eqref{c.estim} and \eqref{c.gnondeg}, $f \in L^{\infty}( [\ut,T] )$, $b \in C_0([\ut,T])$, and $(W_t)_{t \in [0,T]}$ a standard Brownian motion. Moreover, define continuous Gaussian process $(J^{b,f}_t)_{t \in [\ut,T]}$ by
\begin{equation*}
J^{b,f}_t \eqdefl b(t) + \int_{\ut}^t g(t-s)f(s) \ud W_s, \quad t \in [\ut,T].
\end{equation*}
If
$f \neq 0$ a.e., then $
\mathrm{supp}\big(\mathrm{Law}\big[(J^{b,f}_t)_{t \in [\ut,T]}\big]\big) = C_{0}([\ut,T])
$.
\end{lem}

We will prove Lemma \ref{support} by applying the support theorem of Kallianpur \cite[Theorem 3]{kalli}, which states that the support of the law of a continuous Gaussian process is the closure of the associated \emph{reproducing kernel Hilbert space} (for the definition and basic properties, we refer to \cite[pp.\ 93--101]{grena}) in the corresponding path space of continuous functions, equipped with the sup norm. In the case of $J^{b,f}$, the reproducing kernel Hilbert space coincides with the range of a particular integral operator, which we determine first through a slight refinement of the key convolution lemma \cite[Lemma 2.1]{chern}, due to Cherny.

\begin{lem}[Density]\label{density}
Let $\ut \in [0,T)$, $g\in L^2((0,\infty))$, and $f \in L^{\infty}( [\ut,T] )$. If
$g$ satisfies condition \eqref{c.gnondeg} and $f \neq 0$ a.e., then the range of the integral operator $K_f: L^2( [\ut,T] ) \rightarrow C_{0}( [\ut,T] )$, defined by 
\begin{equation}\label{intoper}
 (K_f h)(t) \eqdefl \int_{\ut}^{t} g(t-s)f(s)h(s) \ud s, \quad t \in [\ut,T],
\end{equation}
is dense in $C_0([\ut,T])$.
\end{lem}
\begin{rem}
The fact that $K_f$ maps always (even when $g$ does not satisfy condition \eqref{c.estim}) to $C_0([\ut,T])$ follows from the $L^2$-continuity of the translation $t \mapsto g(t-\cdot)$, see e.g.\ Theorem 9.5 of \cite{rudin}.
\end{rem}

\begin{proof}
Let $h \in C_0([\ut,T])$, $\varepsilon>0$, and denote by $K_1$ the operator defined by \eqref{intoper}, but with $f(s) \equiv 1$. By Lemma 2.1 of \cite{chern}, we know that the assertion holds for $K_1$. Thus, there exists $\tilde{h} \in L^2([\ut,T])$ such that $\|h-K_1 \tilde{h}\|_{\infty}< \varepsilon / 2$, and as $\|K_1\|_{\mathrm{op}} \leq \|g\|_2$ (by Cauchy--Schwarz), we may actually assume that $\tilde{h} \in L^{\infty}([\ut,T])$. 
Since $f \neq 0$ a.e., there exists $\delta>0$ such that the preimage $A_{\delta} \eqdefl f^{-1}([-\delta,\delta])$ satisfies $\lambda(A_{\delta})< \varepsilon^2 / (4 \|g\|^2_2\| \tilde{h}\|^2_{\infty})$. This allows us to define $\hat{h} \in L^{\infty}([\ut,T])$ by $\hat{h}(t) \eqdefl \tilde{h}(t) / f(t) \mathbf{1}_{[\ut,T] \setminus A_{\delta}}(t)$, so that
\begin{equation*}
\|\tilde{h}-f\hat{h} \|_2 = \| \tilde{h}\mathbf{1}_{A_{\delta}}\|_2 \leq \| \tilde{h}\|_{\infty} \lambda(A_{n_0})^{1/2} < \frac{\varepsilon}{2 \|g\|_2}. 
\end{equation*}
Hence $\|K_1 \tilde{h} - K_1 (f\hat{h}) \|_{\infty} < \varepsilon/2$, and finally, since $K_f \hat{h} = K_1(f\hat{h})$, we have
\begin{equation*}
\| h - K_f \hat{h}\|_{\infty} \leq \| h - K_1 \tilde{h} \|_{\infty} + \| K_1 \tilde{h} - K_1 (f\hat{h}) \|_{\infty} < \varepsilon. \qedhere
\end{equation*}
\end{proof}

\begin{proof}[Proof of Lemma \ref{support}]
We may assume that $b(t) \equiv 0$, as the extension to the general case is obvious. The reproducing kernel Hilbert space $\mathscr{H}$ associated to $J^{0,f}$ coincides with the range of $K_f$, see e.g.\ \cite[p.\ 97]{grena}. Further, it is easy to check that by condition \eqref{c.estim}, the covariance function of $J^{0,f}$ is continuous. Hence, by Theorem 3 of \cite{kalli}, $\mathrm{supp}\big(\mathrm{Law}\big[(J^{0,f}_t)_{t \in [\ut,T]}\big]\big)$ is given by the closure of $\mathscr{H}$ in $C_{0}([\ut,T])$, which by Lemma \ref{density} equals $C_{0}([\ut,T])$.
\end{proof}

As the main ingredient in the conditioning argument we are about to use, we apply the following, probably well-known result that establishes Gaussianity of certain conditional laws of stochastic integrals with respect to Brownian motion, when the integrands are independent of the integrator. It can be proved conveniently e.g.\ using \emph{conditional characteristic functions} and \emph{complex Dol\'eans exponentials}. 

\begin{lem}[Gaussianity]\label{Gaussianity} 
Suppose $\mathscr{G}\subset \mathscr{F}$ is a $\sigma$-algebra, $K^1,\ldots,K^d$ are $\mathscr{G}$-measurable random variables, $H^1,\ldots,H^d$ are $\mathscr{G}\otimes\mathscr{B}([u,v])$-measurable processes such that $\int_u^v (H^j_s)^2 \ud s <\infty$ a.s.\ for all $j=1,\dots,d$, and $(W_t)_{t \in [u,v]}$ is a standard Brownian motion independent of $\mathscr{G}$. Then, for a.e.\ $\omega \in \Omega$,
\begin{equation*}
\mathrm{Law}\bigg[ \bigg(K^1 + \int_u^v H^1_s \ud W_s, \ldots, K^d + \int_u^v H^d_s \ud W_s \bigg) \bigg| \mathscr{G}\bigg](\omega) = \mathrm{N}_d(\boldsymbol{\mu}(\omega),\boldsymbol{\Sigma}(\omega)),
\end{equation*}
where $\boldsymbol{\mu} \eqdefl (K^1,\ldots,K^d)$ and $\boldsymbol{\Sigma}_{j,k} \eqdefl \int_u^v H^j_sH^k_s \ud s$ for all $j,k = 1,\ldots,d$.
\end{lem}

Now we are finally ready to prove Theorem \ref{cfsforbss}. The conditioning enters the proof so that, instead of the natural filtration of the process $Z$, we show that $Z$ has CFS with respect to a larger filtration $(\mathscr{G}_t)_{t \in [0,T]}$, given by 
\begin{equation*}
\mathscr{G}_t \eqdefl \sigma \{B_s : s \in (-\infty,t] \} \vee \sigma\{Y_s : s \in [0,T]\} \vee \sigma \{\sigma_s: s \in (-\infty,T]\}.
\end{equation*}
By Corollary 2.9 of \cite{pakka}, this is sufficient.

\begin{proof}[Proof of Theorem \ref{cfsforbss}]
Henceforth, we assume that $B$ is a \emph{standard} Brownian motion, as the extension to the general case is obvious. Fix $\ut \in [0,T)$. Let us consider the continuous process
\begin{equation*}
Z'_t \eqdefl Z_t - Z_{\ut} = \underbrace{Y_t - Z_{\ut} + \int_{-\infty}^{\ut} g(t-s)\sigma_s \ud B_s}_{\eqdefr Y'_t} + \int_{\ut}^t g(t-s) \sigma_s \ud B_s, \quad t \in [\ut,T].
\end{equation*}
(Note that $Y'$ has a continuous modification by condition \eqref{c.estim}.) Denote by $\nu$ the regular $\mathscr{G}_{\ut}$-conditional law of $Z'$ on $C_0([\ut,T])$, and by $\mu^{b,f}$ the law of the Gaussian process $J^{b,f}$ (see Lemma \ref{support}). Our aim is to show that for a.e.\ $\omega \in \Omega$, the laws $\nu(\omega,\cdot)$ and $\mu^{Y'(\omega),\sigma(\omega)}$ coincide.

Recall that these laws are defined on the Borel $\sigma$-algebra $\mathscr{B}(C_0([\ut,T]))$ that is generated as follows.
For any $\boldsymbol{t} = (t_1,\ldots t_d) \in ([\ut,T]\cap \mathbb{Q})^d$ and $\boldsymbol{q}=(q_1,\ldots,q_d) \in \mathbb{Q}^d$, 
define
\begin{equation*}
A_{\boldsymbol{t},\boldsymbol{q}} \eqdefl \bigcap_{j=1}^{d} \pi_{t_j}^{-1}((-\infty,q_j]) \subset C_0([\ut,T]).
\end{equation*}
Then, the family $\mathscr{C}\eqdefl\big\{ A_{\boldsymbol{t},\boldsymbol{q}} : \boldsymbol{t} \in ([\ut,T]\cap \mathbb{Q})^d, \boldsymbol{q} \in \mathbb{Q}^d, d=1,2,\ldots \big\}$ is a countable $\pi$-system that generates $\mathscr{B}(C_0([\ut,T]))$.

Accordingly, it suffices to show that for any $A_{\boldsymbol{t},\boldsymbol{q}} \in \mathscr{C}$, and for a.e.\ $\omega \in \Omega$,
\begin{equation}\label{marginals}
\nu(\omega,A_{\boldsymbol{t},\boldsymbol{q}}) = \mu^{Y'(\omega),\sigma(\omega)}(A_{\boldsymbol{t},\boldsymbol{q}}).
\end{equation}
But by Lemma \ref{Gaussianity} and the disintegration theorem \cite[Theorem 6.4]{kalle}, we have a.s.\
\begin{equation*}
\begin{split}
\nu(\cdot,A_{\boldsymbol{t},\boldsymbol{q}}) & = \mathbf{E}\big[\mathbf{1}_{\{Z'_{t_1}\leq q_1,\ldots Z'_{t_d} \leq q_d\}}\big|\mathscr{G}_{\ut}\big] \\
& = \int_{-\infty}^{q_1}\cdots \int_{-\infty}^{q_d} \phi_d(x_1,\ldots,x_d;\boldsymbol{\mu},\boldsymbol{\Sigma}) \ud x_1\cdots \ud x_d,
\end{split}
\end{equation*}
where $\boldsymbol{\mu} \eqdefl (Y'_{t_1},\ldots,Y'_{t_d})$ and $\boldsymbol{\Sigma}_{j,k} \eqdefl \int_{\ut}^{t_j \wedge t_k} g(t_j - s) g(t_k -s) \sigma^2_s \ud s$ for all $j,k = 1,\ldots,d$, implying \eqref{marginals}.

Since we assumed that $\proba\big[ \lambda(\{ t \in[0,T]: \sigma_t = 0 \}) = 0 \big]=1$, we have now by Lemma \ref{support} for a.e.\ $\omega \in\Omega$,
\begin{equation*}
\mathrm{supp}(\nu(\omega,\cdot)) = \mathrm{supp}\big( \mu^{Y'(\omega),\sigma(\omega)}\big) = C_0([\ut,T]),
\end{equation*}
which implies the assertion.
\end{proof}

\section*{Acknowledgements}

I would like thank Ole E.\ Barndorff-Nielsen for asking, and Mark Podolskij for re-asking the question, whether Brownian semistationary processes have conditional full support; and for commenting on this paper. Moreover, I am grateful to Tommi Sottinen for discussions and comments. 
The research in the paper was supported by the Academy of Finland (project 116747) and the Finnish Cultural Foundation.


\begin{thebibliography}{10}

\bibitem{barnd2}
{\sc O.~E. Barndorff-Nielsen, J.~M. Corcuera, and M.~Podolskij}, {\em Limit
  theorems for functionals of higher order differences of {B}rownian
  semi-stationary processes}.
\newblock Preprint, 2009.

\bibitem{barnd3}
\leavevmode\vrule height 2pt depth -1.6pt width 23pt, {\em Multipower variation
  for {B}rownian semistationary processes}.
\newblock Preprint, 2009.

\bibitem{barnd}
{\sc O.~E. Barndorff-Nielsen and J.~Schmiegel}, {\em Brownian semistationary
  processes and volatility/intermittency}.
\newblock Preprint, 2009.

\bibitem{chern}
{\sc A.~Cherny}, {\em Brownian moving averages have conditional full support},
  Ann. Appl. Probab., 18 (2008), pp.~1825--1830.

\bibitem{delba}
{\sc F.~Delbaen and W.~Schachermayer}, {\em A general version of the
  fundamental theorem of asset pricing}, Math. Ann., 300 (1994), pp.~463--520.

\bibitem{grena}
{\sc U.~Grenander}, {\em Abstract inference}, Wiley, New York, 1981.

\bibitem{guaso2}
{\sc P.~Guasoni, M.~R{\'a}sonyi, and W.~Schachermayer}, {\em Consistent price
  systems and face-lifting pricing under transaction costs}, Ann. Appl.
  Probab., 18 (2008), pp.~491--520.

\bibitem{kalle}
{\sc O.~Kallenberg}, {\em Foundations of modern probability}, Springer, New
  York, 2nd~ed., 2002.

\bibitem{kalli}
{\sc G.~Kallianpur}, {\em Abstract {W}iener processes and their reproducing
  kernel {H}ilbert spaces.}, Z. Wahrscheinlichkeitstheorie und Verw. Gebiete,
  17 (1971), pp.~113--123.

\bibitem{pakka}
{\sc M.~S. Pakkanen}, {\em Stochastic integrals and conditional full support}.
\newblock Preprint (available as arXiv:0811.1847), 2009.

\bibitem{rudin}
{\sc W.~Rudin}, {\em Real and complex analysis}, McGraw-Hill, New York,
  3rd~ed., 1987.

\end{thebibliography}
\end{document}